%% file: ComputingTheLayerPoset_arXiv_v2.tex
\documentclass[reqno, a4paper
]{amsart} % 10pt is standard

\usepackage[latin1]{inputenc}
\usepackage{subfigure}

\usepackage{tikz}
\usetikzlibrary{external}
\tikzexternaldisable
\usetikzlibrary{arrows}

\usepackage{amssymb} 
\usepackage{amsmath} %-> automatisch geladen von amsart
\usepackage{amsfonts} %-> automatisch geladen von amsart
\usepackage{amsthm}
\usepackage{upref} % \ref always produces roman output
\usepackage{graphicx}
\usepackage{color}
\usepackage{url}
\usepackage{paralist} % bessere enumerate und itemize Umgebungen
\usepackage{algorithmicx}
\usepackage{algpseudocode}
\usepackage{caption} % allows linebreaks in captions of figures

\usepackage[british]{babel}

\usepackage{hyperref}

\hypersetup{pdfauthor={},pdftitle={}, 
  pdfkeywords={} }

\newtheorem{Theorem}{Theorem}%[section]

\newtheorem{Lemma}[Theorem]{Lemma}

\theoremstyle{definition}

\newtheorem{Example}[Theorem]{Example}

\theoremstyle{remark}
\newtheorem{Remark}[Theorem]{Remark}

\newcommand{\forOL}[1]{\textbf{for all} #1 \textbf{do}} % a one-line for
\newcommand{\ifOL}[1]{\textbf{if} #1 \textbf{then}} % a one-line if
\newcommand{\elseOL}[1]{\textbf{else}} % a one-line if
\newcommand{\indentOL}{\quad}

\DeclareMathOperator{\LG}{LG} % layer group
\DeclareMathOperator{\FM}{FM} %
\DeclareMathOperator{\DAM}{DAM} % 

\newcommand{\abs}[1]{\left|#1\right|}
\DeclareMathOperator{\rank}{rk}

\newcommand{\st}{s.\,t.\ } % such that
\newcommand{\ie}{\textit{i.\,e.\ }} % id est
\newcommand{\eg}{\textit{e.\,g.\ }} % exemplum gratum

\newcommand{\Z}{\mathbb{Z}}

\newcommand{\R}{\mathbb{R}}
\newcommand{\CC}{\mathbb{C}}

\newcommand{\Acal}{\mathcal{A}}

\newcommand{\Ccal}{\mathcal{C}}

\newcommand{\Gcal}{\mathcal G}
\newcommand{\Lcal}{\mathcal L}

\tikzstyle{bonn1colour}=[red!90!black]
\tikzstyle{bonn2colour}=[cyan!78!black]
\tikzstyle{bonn3colour}=[green!80!black]
\tikzstyle{bonn4colour}=[blue!90!black]

\definecolor{darkgreen}{rgb}{0,0.8,0}

\newcommand{\colourBonnA}{red}
\newcommand{\colourBonnB}{cyan}
\newcommand{\colourBonnC}{darkgreen}
\newcommand{\colourBonnD}{blue}

\newlength{\BonnLatticeFlatsLength}

\newlength{\topXPosetLayers}
\newlength{\topshiftXPosetLayers}
\newlength{\middleXPosetLayers}
\newlength{\distAPosetLayers}
\newlength{\distBPosetLayers}
\author{Matthias Lenz} 
\address{Universit\'e de Fribourg, D\'epartement de Math\'ematiques, 1700 Fribourg, Switzerland}
\email{maths@matthiaslenz.eu}

\thanks{%
The author was supported by a Swiss Government Excellence Scholarship for Foreign
Scholars and subsequently by a fellowship within the postdoc programme of the German
Academic Exchange Service (DAAD)%
}

\title%[]
{%
 Computing the 
 poset of layers of a toric arrangement
}

\date{\today}
\address{%
}


\begin{document}

\begin{abstract}
  A toric arrangement is an arrangement of subtori of codimension one in a real or complex torus.
  The poset of layers is the set of connected components of 
  non-empty intersections of these subtori, partially ordered by reverse inclusion.
  In this note we present an algorithm that computes this poset  in the central case.
  \end{abstract}

\maketitle

\section{Introduction}

A toric arrangement is an arrangement of subtori of codimension one in a  real or complex torus 
(\eg \cite{ardila-castillo-henley-2015,callegaro-delucchi-2017,delucchi-dantonio-2016,deconcini-procesi-2005,ehrenborg-readdy-slone-2009,lawrence-2011,moci-tutte-2012}).
Studying these objects is a natural step beyond 
the classical theory 
of hyperplane arrangements.
Both, the topology of the complement as well as the combinatorial structure of toric
arrangements are active areas of research. 
 Hyperplane arrangements have been studied intensively
  using methods from 
 combinatorics, algebra, algebraic geometry, 
 and topology 
\cite{orlik-terao-1992, stanley-2007}.
 While matroids capture combinatorial information about hyperplane arrangements,
 arithmetic mat\-roids
appear as one of the prominent combinatorial structures connected with toric arrangements \cite{branden-moci-2014,moci-adderio-2013}.

Toric arrangements are particularly important due to their connection with the problem of
counting lattice points in polytopes.
This was implicitly discovered in the 1980s 
by researchers working on splines \cite{BoxSplineBook}
  such as  Dahmen and Micchelli.
 It  was recently made more explicit and put in a broader context by De~Concini, Procesi, Vergne, and others 
 \cite{concini-procesi-book}.
 Delucchi and Riedel have recently investigated
 group actions on semimatroids \cite{delucchi-riedel-2015},
 a structural framework
 to study  toric arrangements and generalizations (\eg ``toric pseudoarrangements'').

 Given a central hyperplane arrangement, represented by the normal vectors of the hyperplanes, % 
 it is relatively easy to come up with an algorithm to calculate its intersection lattice (Section~\ref{Section:HyperplaneArrangement}).
 The analogous problem for the poset of layers of a toric arrangement is more difficult.
 In Section~\ref{Section:LG} we present such an algorithm that
 uses some of the structural results of Delucchi and Riedel.

 In the case of central hyperplane arrangements,
 it is known that the intersection lattice is isomorphic to the lattice of flats of the underlying matroid.
 The lattice of flats of a matroid  is a geometric lattice, also in the the non-representable case.
 Vice versa, every geometric lattice is isomorphic to the lattice of flats of a simple matroid (\eg \cite{stanley-2007}).
 Much less is known in the case of toric arrangements:  their poset of layers is  
   always a quotient of a geometric semilattice \cite{delucchi-riedel-2015}.
 A matroidal analogue (\eg a ``poset of flats'' of an arithmetic matroid) is currently not known.

 It is known that 
 the intersection lattice has nice properties and carries a lot of information about a hyperplane arrangement.
 In the case of toric arrangements, things are less clear.
 On the positive side,  
 both, the
 number of connected components  of the complement of a real hyperplane arrangement
 and of a real toric arrangement are determined by the respective characteristic polynomials, which depend only on the posets
  \cite{ehrenborg-readdy-slone-2009,zaslavsky-1975}.
 In both cases, the characteristic polynomials also capture the Betti numbers of the complement of the complex  arrangements
  \cite{deconcini-procesi-2005,orlik-solomon-1980}.
 On the other hand, there are still several open questions that are currently being investigated.
  For example, it is known that geometric lattices are shellable, 
  but it is an open problem if the poset of layers is shellable.
  In the case of root systems, this has recently been established \cite{delucchi-girard-paolini-2017}.
 It is also not clear to what extent the poset of layers determines topological data
 such as the integer cohomology ring of the complement
 of a complex toric arrangement (see the discussion in
 \cite[Section~2.3.3]{callegaro-delucchi-2017} and
 \cite{lenz-unique-2017} for some recent progress).

 The availability of an algorithm  will help researchers to get a better intuition
 about the poset of layers of a toric arrangement and help to check conjectures on larger examples.
 A sage \cite{sage-80} implementation of this algorithm is forthcoming.

 The table in Figure~\ref{Figure:ContinuousDiscrete} provides an overview of how the structures mentioned 
 so far are related to each other.
 The connections with equivariant cohomology and $K$-theory are explained in \cite{deconcini-procesi-vergne-2013}.
 The descriptions as continuous and discrete objects 
 are due to the connection with volumes and the number of integer points in polytopes and 
 are used in the  literature on zonotopal algebra (\eg \cite{concini-procesi-book,holtz-ron-2011,lenz-arithmetic-2016}).

\begin{figure}[t]
 \begin{tabular}{l|l} 
    continuous  & discrete  \\\hline
    hyperplane arrangement & toric arrangement \\
    matroid    &  arithmetic matroid  \\
    geometric lattice &   (poset of layers) \\
    volume of a polytope & no.~of integer points \\
    cohomology & $K$-theory 
 \end{tabular}
 \caption{Some of the structures that are related to hyperplane arrangements and toric arrangements}
 \label{Figure:ContinuousDiscrete}
\end{figure}

\subsection*{Acknowledgements}
  The author would like to thank Emanuele Delucchi, Henri M\"uhle, and Sonja Riedel for some helpful conversations.

\section{Arrangements}

 \subsection{Setup}
 Throughout this note   
 $X=(x_1,\ldots, x_N)$ denotes a list (or sequence) of $N$ vectors that are contained in $\Z^d$ (for toric arrangements) or
 in a finite dimensional vector space (for hyperplane arrangements).
 We suppose that $X$ spans the ambient vector space. % 
 Sometimes, we will consider $X$ as a $(d\times N)$-matrix with columns $x_1,\ldots x_N$.
 For $S\subseteq [N]$, $X[S]$ 
 denotes  the submatrix of $X$ whose columns are indexed by $S$
 and $\rank(S)$ denotes the rank of $S$ in the matroid theoretic sense, which is equal to the rank of $X[S]$ in the linear algebra sense.
 
 \subsection{Arrangements}

 In this subsection we will introduce three types of arrangements: finite hyperplane arrangements, (infinite) periodic
  hyperplane arrangements, and toric arrangements (see \eg \cite[Section~2]{delucchi-riedel-2015}).

 For $S\subseteq [N]$  and $k\in \Z^S$,  
 we define the subspace
 \begin{equation}
     H(S,k) := \{   \alpha \in \R^d :  x_i^t \alpha = k_i \text{ for all } i \in S\}.
 \end{equation}

 Now we define
 \begin{align*}
 \Acal_H(X)&:= \{ H(\{i\},0) : i \in [N]\}, \text{the \emph{central hyperplane arrangement},} 
 \label{eq:hyperplanearrangement}
 \\
 \Acal_P(X)&:= \{ H(\{i\},k) : i \in [N], k\in \Z  \}, \text{ the \emph{periodic hyperplane arrangement}, and} \\ 
 \Acal_T(X)&:= \{  H/\Z^d : H \in \Acal_P(X)\},  \text{ the \emph{central toric arrangement} in the torus } \R^d/\Z^d.
 \end{align*}
 If we translate some elements of $\Acal_H$ or $\Acal_T$, we obtain corresponding affine arrangements.

  Using the exponential map,
  the real (or compact) torus $\R^d / \Z^d$ is isomorphic to 
   $(S^1)^d$.
  As usual, $S^1 := \{ z \in \CC : \abs{z} = 1 \}$. 
  In this setting, it is common to describe a central toric arrangement as an arrangement of kernels of characters:
 every $v=(v_1, \ldots, v_d)\in \Z^d$ determines a \emph{character} of the torus, \ie a map 
 $\chi_v: (S^1)^d\to S^1$ via $\chi_v( (\alpha_1,\ldots, \alpha_d)):= \alpha_1^{v_1} \cdots  \alpha_d^{v_d}$.
 Each $x\in X$ defines a
   (possibly  disconnected) hypersurface 
  $S_x := \{ \alpha \in (S^1)^d  : \chi_x(\alpha) = 1 \}$.  

 It is also interesting to study toric arrangements in the complex (or algebraic) torus 
 $(\CC^*)^d$, which are defined in an analogous way. However, in this note we are mainly interested in  
 the posets of layers of central toric arrangements.
 These posets only depend on the list $X$ and 
 they are the same in the real and the complex case.

\begin{Example}
  \label{Example:Bonn}
 Let $X=({\color{\colourBonnA} ( 2 , 0  ) }, 
  {\color{\colourBonnB}  ( 0,   1  ) }, 
   {\color{\colourBonnC} ( 1,   -1  ) }, 
    {\color{\colourBonnD} ( 2,   2 )  }  )$.  
   The arrangements defined by the list $X$ are shown in Figure~\ref{Figure:BonnArrangements}.
   \end{Example}
 \begin{figure}[t]
 \begin{center}
\input{./Bonn_HyperplaneArrangement.tikz}
\input{./Bonn_ToricArrangement.tikz}
% 
% \tikzexternalenable
% \tikzsetnextfilename{Bonn_Donut}
% 
   \input{./sketch_Bonn_toric_arrangement.tikz} 
% \tikzexternaldisable
% 
% 
 \end{center}
 \caption{A hyperplane arrangement and a toric arrangement in $(\R/\Z)^2$ and in $(S^1)^2$.
 Both arrangements are defined by the list $X$ in Example~\ref{Example:Bonn}.
\\
{\footnotesize
These image were created by the author in 2014. They are made available under the Creative Commons Attribution 4.0 International License. 
To view a copy of this license, visit \url{http://creativecommons.org/licenses/by/4.0/}.
}
}
 \label{Figure:BonnArrangements}
\end{figure}

 \subsection{Posets}

 Let $\Acal$ be a (possibly infinite) hyperplane arrangement.
 We define  $\Lcal(\Acal)$,
 the \emph{poset of intersections} of $\Acal$ as the set of 
 non-empty intersections of some of the hyperplanes in $\Acal$, ordered by reverse inclusion.
 In the finite case, this is a geometric lattice, the \emph{intersection lattice} of the hyperplane arrangement.
 In the periodic case, it is a geometric semilattice in the sense of Wachs and Walker \cite{wachs-walker-1986}.

 Let $\Acal_T$ be a toric arrangement. 
A \emph{layer} is a connected component of a non-empty intersection of some of the subtori in $\Acal_T$.
The \emph{poset of layers} $\Ccal(\Acal_T)$ of $\Acal_T$ is 
the set of its layers,  ordered by reverse inclusion
(cf.~Figure~\ref{Figure:BonnPosetLayer}).
This is an analogue of the intersection lattice of a hyperplane arrangement.

If $\Acal_T$ is a toric arrangement and $\Acal_P$ the corresponding periodic hyperplane arrangement, 
$\Ccal(\Acal_T)$ is a quotient of $\Lcal(\Acal_P)$ under an action induced by $\Z^d$
(cf.~\cite[Remark~2.5]{delucchi-riedel-2015}).

 \begin{figure}[t]
\begin{center}
\newcommand{\setBonnLatticeFlatsScaling}{% 
 \setlength{\BonnLatticeFlatsLength}{0.8cm} % 
}
\newcommand{\setBonnLayerPosetScaling}{% 
  \setlength{\middleXPosetLayers}{-3.7cm} % 
  \setlength{\distAPosetLayers}{1.3cm}
  \setlength{\distBPosetLayers}{1.5cm}   
  \setlength{\topXPosetLayers}{-3.5cm} % 
  \setlength{\topshiftXPosetLayers}{1.25cm}
}
  \input{./Bonn_LatticeFlats.tikz}   \hspace*{-9mm}
  \input{./Bonn_PosetLayers.tikz} 
\end{center}
\caption{The lattice of flats of the matroid and the poset of layers of the toric arrangement corresponding to Example~\ref{Example:Bonn}.
 Note that for the toric arrangement, $x_1$ and $x_4$ both define two one-dimensional layers.
 The labels of the zero-dimensional layers are the same as the ones in Figure~\ref{Figure:BonnArrangements}.}  
\label{Figure:BonnPosetLayer}
\end{figure}
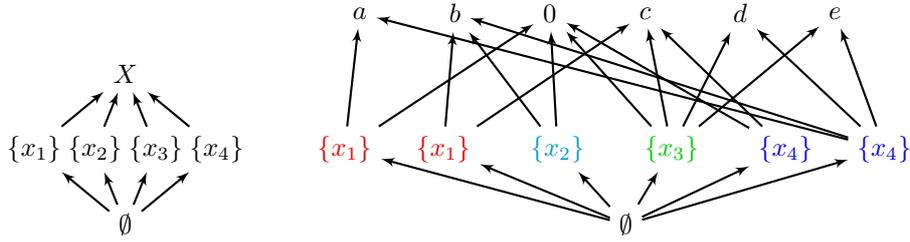

 Both,  the poset of intersections  and the poset of layers are graded posets. In both cases, the rank of an intersection/layer
 $K$ can be defined as $d-\dim(K)$.
 
 Any finite poset can be represented by a directed acyclic graph as follows: the set of vertices is equal to the ground set of the poset
 and two vertices $x,y$ are connected by an arc $(x,y)$ if and only if $y$ covers $x$ in the poset.
 We will refer to this graph as the 
 \emph{Hasse diagram} of the poset.
 Vice versa, every directed acyclic graph represents a poset.

 \subsection{Matroidal structures}
 \label{Subsection:MatroidalStructures}
 
 The underlying matroid captures a lot of information about a finite, central hyperplane arrangement 
 \cite{orlik-terao-1992, stanley-2007}.
  For affine hyperplane arrangements,  semi-matroids play a similar role \cite{ardila-2007}.
 In the case of central toric arrangements,
  arithmetic matroids capture combinatorial and topological information.
   An arithmetic matroid is a matroid together with a multiplicity function $m$ that assigns
  a positive integer (the multiplicity) to each subset of the ground set
  \cite{branden-moci-2014,moci-adderio-2013,moci-tutte-2012}.
  In the representable case, \ie when the arithmetic matroid is given by a list of 
  vectors in a lattice, the multiplicity of an independent set $S$  is equal to the greatest common divisor 
  of all minors
   of size $\abs{S}$ of the matrix $X[S]$ (this is essentially \cite[Theorem 2.2]{stanley-1991}).
  For arbitrary sets $S\subseteq [N]$, $m(S) = \gcd({m(B) : B \subseteq  S \text{ and } \abs{B} = \rank(B) = \rank(S)})$ holds
   (cf.~\cite[p.~344]{moci-adderio-2013}, see also \cite[Lemma~8.4]{lenz-ppcram-2017}).

  Delucchi and Riedel recently introduced $G$-semimatroids that capture   
  the combinatorics of periodic arrangements of hyperplanes (and more general hypersurfaces) and their quotients \cite{delucchi-riedel-2015}.

 \section{Hyperplane arrangements and their intersection lattices}
 \label{Section:HyperplaneArrangement}
 
 In this section we will present an algorithm that can be used to calculate 
 the intersection lattice of a central hyperplane arrangement.
 It
 is shown in  Figure~\ref{Figure:AlgorithmHA}.
 The algorithm in  Section~\ref{Section:LG} that calculates the poset of layers of a toric arrangement
 is an extension of this algorithm.
 
 Note that the algorithm given here uses only the matroid data of the list $X$. 
 Given a matroid on $N$ elements as input, it would calculate its lattice of flats.

\begin{figure}[ht]
\begin{algorithmic}
\State \textbf{INPUT:} A list $X = (x_1,\ldots, x_N)$ of elements of some vector space $V$
\vspace*{1mm}
\State \textbf{OUTPUT:} A directed graph, the Hasse diagram of the intersection lattice of the central hyperplane arrangement defined by $X$
\vspace*{2mm}
\State \textbf{ALGORITHM:}
\State Create a digraph $\Gcal$ with vertices $ \{ v_S : S\subseteq [N] \} $ and no edges
\State \forOL{pairs $(T,S)\subseteq [N]^2$ with $T=S\setminus \{s\}$ for some $s\in S$} 
\State\indentOL
    \ifOL{$\rank(S) = \rank(T)$ } % 
    \quad Add a blue arc $(v_T, v_S)$ to $\Gcal$
\State\indentOL
    \elseOL 
   \State 
   \quad
   Add a black arc $(v_T, v_S)$ to $\Gcal$
 \State Contract the blue arcs in $\Gcal$ 
 \State \Return  $\Gcal$
\end{algorithmic}
\caption{An algorithm that calculates the intersection lattice of a central hyperplane arrangement}
\label{Figure:AlgorithmHA}
\end{figure}

\section{Layer groups and the computation of the poset of layers}
\label{Section:LG}

\subsection{Layer groups}

  As usual, let $X$ be a $(d\times N)$-matrix with integer entries. 
  Let $S\subseteq [N]$.
  Then $X[S]$ denotes the submatrix of $X$ whose columns are indexed by $S$.
  Following \cite{delucchi-riedel-2015} we define the  groups
  \begin{align}
    W(S) &:=  X[S]^t \R^d \cap \Z^S,
    \\
    I(S) &:= X[S]^t \Z^d, \\
    Z(S) &:= \Z^S / I(S), \text{ and} \\
    \LG(S) &:= W(S) / I(S). 
  \end{align}
 The \emph{layer group} $\LG(S)$ is the torsion subgroup of $Z(S)$ \cite[Lemma~2.11]{delucchi-riedel-2015}.
  
 It is easy to see that $W(S) = \{ k \in \Z^S : H(S,k) \neq \emptyset\}$ \cite[Remark~2.4]{delucchi-riedel-2015}.
 Let us consider the map from $W(S)$ to the connected components of $\bigcup_{k\in \Z^S} H(S,k)$ that is defined 
  by $\varphi_S(k) := H(S,k)$.
  This map is bijective
  \cite[Lemma~2.6]{delucchi-riedel-2015}.
  Passing over to the quotient, we obtain a bijective map $\bar\varphi_S$ from $\LG(S)$ to the set of layers 
  of the toric arrangement that are defined by $S$
 \cite[Remark~2.31]{delucchi-riedel-2015}.
 This justifies the name layer group.
  We will write $C(S)$ to denote  the set of layers defined by $S$. 
  Note that all elements of $C(S)$ have dimension $d-\rank(S)$.

 The following two lemmas show that there are nice maps between different layer groups.
 Let $s\in S\subseteq [N]$ and let $T:=S\setminus \{s\}$.
 Let  $\pi_{S,T} : \Z^S \to \Z^T$ denote the projection that forgets the $s$ coordinate.

 \begin{Lemma}
  \label{Lemma:Ipi}
    $ \pi_{S,T} ( I(S) ) = I(T)$.
  \end{Lemma}

 \begin{Lemma}
  \label{Lemma:Wpi}
    $ \pi_{S,T} ( W(S) ) \subseteq W(T)$.
  \end{Lemma}

 Lemma~\ref{Lemma:Ipi} implies that $\pi_{S,T}$ induces a map
 $\bar\pi_{S,T} : Z(S) \to Z(T) $.
 Lemma~\ref{Lemma:Ipi} and Lemma~\ref{Lemma:Wpi} imply that 
 $\pi_{S,T}$ induces a map
 $\bar\pi_{S,T} : \LG(S) \to \LG(T) $.
 
 \begin{Lemma}
 \label{Lemma:Injection}
   Let $ \rank(S) = \rank(T)$.
   Then
   $ \bar \pi_{S,T} : \LG(S) \hookrightarrow \LG(T)$.

 \end{Lemma}

 \begin{Lemma}
  \label{Lemma:Surjection}
        Let $ \rank(S) > \rank(T)$.
        Then  % 
        $ \bar\pi_{S,T} : \LG(S) \twoheadrightarrow \LG(T)$.
 \end{Lemma}

 This section does not contain any proofs. They are postponed to 
 Section~\ref{Section:Proofs}.

 \begin{Remark}
 Several structures that are related to the layer groups have appeared in the literature.    
 This includes the groups
  $\DAM(S) := (X[S] \R^S \cap \Z^d) /  X[S] \Z^S$
  and  $\FM(S) := \Z^d / X[S] \Z^S$
  in the context of representable arithmetic matroids \cite{moci-adderio-2013}
  and representable matroids over $\Z$ \cite{fink-moci-2016}.
  If $m(S)$ denotes the multiplicity of the set $S$ in the arithmetic matroid defined by $S$, 
  one can show that
  $\abs{\LG(S)} = m(S)  =  \abs{\DAM(S)} $. This is equal to the cardinality of the torsion part of $\FM(S)$ 
  (see also \cite[Remark~2.12 and Theorem~2.17]{delucchi-riedel-2015}).    
  For our purposes the layer groups are most suitable:
  the layer groups and the $\DAM$-groups encode similar information and they are in some sense dual to each other.
  However, our algorithm requires the existence of canonical maps between the layer groups that have certain properties (Lemma~\ref{Lemma:Injection}
  and Lemma~\ref{Lemma:Surjection}). There are also canonical maps between the $\DAM$-groups, but the arrows point in the other direction,
  hence these groups cannot be used directly in the algorithm.
\end{Remark}

\subsection{The algorithm}

 \begin{Theorem}
    \label{Theorem:ItWorks}
    The algorithm in Figure~\ref{Figure:AlgorithmTA} correctly calculates the poset of layers of the central toric arrangement defined by the list $X$.
 \end{Theorem}

 \begin{figure}[ht]
\begin{algorithmic}
\State \textbf{INPUT:} A list $X = (x_1,\ldots, x_N)$ of elements of $\Z^d$ that spans $\R^d$
\vspace*{1mm}
\State \textbf{OUTPUT:} A directed graph, the Hasse diagram of the poset of layers of the toric arrangement defined by $X$
\vspace*{2mm}
\State \textbf{ALGORITHM:}
\State   For each $S\subseteq [N]$, calculate $\LG(S)$
\State Create a digraph $\Gcal$ with vertices $  \{   v_{S,g} :  S \subseteq [N] \text{ and } g\in  \LG(S)\}$ and no edges
\State \forOL{pairs $(T,S)\subseteq [N]^2$ with $T=S\setminus \{s\}$ for some $s\in S$} 
\State \indentOL
      \ifOL{$\rank(S) = \rank(T)$ } % 
\State \indentOL \indentOL
      \forOL{$h\in \LG(S)$} \quad 
          add a blue arc
          $(v_{T,\pi_{S,T}(h)},   v_{S,h})$ to $\Gcal$
\State \indentOL  
  \elseOL 
   \State  % 
\State \indentOL \indentOL 
  \forOL{$h\in \LG(S)$} \quad 
       add a black arc
       $(v_{T,\pi_{S,T}(h)}),   v_{S,h})$ to $\Gcal$
\State Contract the blue arcs in $\Gcal$ 
\State \Return  $\Gcal$
\end{algorithmic}
\caption{An algorithm that calculates the poset of layers of a central toric arrangement}
\label{Figure:AlgorithmTA}
\end{figure}

\begin{Remark}
 The algorithms presented in this note are not very efficient.
 Recall that by differentiating the binomial theorem, one obtains the formula $\sum_{k=0}^N k \binom{N}{k} = N \cdot 2^{N-1}$.
 This is the number of iterations of the outer for-loop in the algorithm.
 Hence a lower bound for the running time of both algorithms is $\Omega(N \cdot 2^{N-1})$.

   Let $M\in \Z$ be an upper  bound for the absolute values of the entries of the matrix $X$.
   Since $\abs{\LG(S)}$ is equal to the multiplicity of the set $S$ in the corresponding arithmetic matroid,
   it follows from the results mentioned in Subsection~\ref{Subsection:MatroidalStructures} that the size of a layer group is bounded
   from above by the maximal value that the determinant of a square submatrix of $X$ may attain.
   By the Leibniz formula, this is bounded from above by
   $d! M^d$.
   This bounds the number of iterations of the two inner for-loops.
   In total, an upper bound for the running time of the algorithm is
     $ O(  N \cdot  2^{N-1} \cdot d! \cdot M^d)$.
   Hence if we fix $d$ and $N$, it is polynomial in $M$.
   In the totally unimodular case or for hyperplane arrangements, the running time is in
   $ \Theta( N \cdot 2^{ N-1 } )$. 
 \end{Remark}

\subsection{Towards a generalization}

 It would be interesting to generalize our algorithm to affine toric arrangements or more generally, 
 to some of the arrangements appearing in the context of $G$-semimatroids \cite{delucchi-riedel-2015}.
 We believe that such a generalization may be possible, but leave it to the reader to construct it.
 As a first step, we present a generalization of the lattice $W(S)$ to the affine case.     
 Let $X$ be a $(d\times N)$-matrix and $c\in \R^N$.
 This data determines an affine toric arrangement in the real or complex torus.
 Let $S\subseteq [N]$ and $k\in \Z^S$. This defines the subspace
  \begin{align}
      H(S,k) = \{ \alpha \in \R^d:  X[S]^t \alpha = k + c_S   \}.
  \end{align}
 Here, $c_S$ denotes the vector $(c_i)_{i \in S} \in \R^S$.
 Delucchi and Riedel  \cite{delucchi-riedel-2015} define
    $W (S) := \{ k \in  \Z^S : H(S, k) \neq \emptyset \}$.
    This implies that in the affine case, $W(S)$ is an affine lattice:
    \begin{align}
      W(S) &= \{ k \in \Z^S :   (X[S]^t)^{-1}(k + c_S) \neq \emptyset \}   
      \\
      &= \{ k \in \Z^S :  \exists u\in \R^d \text{ \st }  X[S]^t u  = k + c_S \}
      \\
      &= (X[S]^t \R^d - c_S) \cap \Z^S .     
    \end{align}

  \section{Proofs}
  \label{Section:Proofs}

 \begin{proof}[Proof of Lemma~\ref{Lemma:Ipi}]
  This is true because $\pi_{S,T} (X[S]^t \alpha) = X[T]^t \alpha$ for all $\alpha\in \Z^d$.
 \end{proof}

 \begin{proof}[Proof of Lemma~\ref{Lemma:Wpi}]
%    Suppose that the matrix $X[S]$ consists of the columns $s_1,\ldots, s_{k+1}\in \Z^d$ and that $X[T]$ consists of the columns
%    $s_1,\ldots, s_k$. 
   Let us denote the columns of the matrix $X[S]$ by $s_1,\ldots, s_{k+1}\in \Z^d$. Then we may assume that  $X[T]$ consists of the columns
   $s_1,\ldots, s_k$.

 Let $z=(z_1,\ldots, z_{k+1}) \in \Z^{k+1}$. 
 By definition, $z \in W(S)$  if and only if there exists $\alpha \in \R^d$ \st
  $s_1^t \alpha = z_1, \ldots, s_{k+1}^t \alpha = z_{k+1}$.

   Let $z'=(z'_1,\ldots, z'_{k}) \in \Z^{k}$. By definition, 
 $z \in W(T)$  if and only if there exists $\beta \in \R^d$ \st
  $s_1^t \beta = z'_1, \ldots, s_{k}^t \beta = z_{k}$.
  An $\alpha$ (as described above) must satisfy the same equations as a $\beta$ and one additional equation. 
  Hence $\Pi_{S,T}(W(S)) \subseteq W(T)$.
 \end{proof}
  
  Note that   if $\rank(S)=\rank(T)$,
   $\pi_{S,T}(W(S)) \subseteq W(T)$, but in general, equality does not hold:  by   \eqref{equation:dependencyZ} below
  for $(z_1,\ldots, z_k)\in W(T)$, the last coordinate of the preimage is equal to 
  $\sum_{i=1}^k \lambda_i z_i$. This is not necessarily an integer.

 \begin{proof}[Proof of Lemma~\ref{Lemma:Injection}]
   As before, let us denote the columns of the matrix $X[S]$ by $s_1,\ldots, s_{k+1}\in \Z^d$. Then we may assume that  $X[T]$ consists of the columns
   $s_1,\ldots, s_k$. 
% 
%   $s_1,\ldots, s_k$. 
  By assumption, there exist $\lambda_1,\ldots, \lambda_k \in \R$ \st 
  \begin{equation}
   \label{eq:dependency}
  s_{k+1} = \sum_{i=1}^k \lambda_i  s_i.
  \end{equation}
 Let $z=(z_1,\ldots, z_{k+1}) \in W(S)$. 
 By definition, $z \in W(S)$  if and only if there exists $\alpha \in \R^d$ \st
  $s_1^t \alpha = z_1, \ldots, s_{k+1}^t \alpha = z_{k+1}$.
  By \eqref{eq:dependency}, the last condition implies
  \begin{equation}
  \label{equation:dependencyZ}
    s_{k+1}^t \alpha =  \sum_{i=1}^k \lambda_i  s_i^t \alpha = \sum_{i=1}^k \lambda_i z_i = z_{k+1}.
  \end{equation}
  By Lemma~\ref{Lemma:Wpi}, we can restrict $\pi_{S,T}$ to a map $W(S)\to W(T)$ that we will also denote by $\pi_{S,T}$.
  Suppose $\pi_{S,T}(z) = (0,\ldots, 0) $. 
  Since $\pi_{S,T}$ is a projection map, this implies $z=(0,\ldots, 0,z_{k+1})$.
  Using \eqref{equation:dependencyZ} we can deduce $z_{k+1}=0$. Hence $z=0$ and $\pi_{S,T} : W(S)\to W(T)$ is injective.
  Using Lemma~\ref{Lemma:Ipi}, we can deduce that $ \bar\pi_{S,T} : \LG(S) \to \LG(T)$ is injective as well.
\end{proof}

 \begin{proof}[Proof of Lemma~\ref{Lemma:Surjection}]
   It is known that
   $\LG(S)$ is the torsion part of $Z(S)$ and the rank of $Z(S)$ is equal to the nullity of $S$
     \cite[Lemma~2.11]{delucchi-riedel-2015}.
  As $ \rank(S) =  \rank(T) +1 $, $Z(S) \cong \Z^{\eta} \oplus \LG(S)$ and $Z(T) \cong \Z^{\eta} \oplus \LG(T)$, where 
  $\eta= \abs{S}- \rank(S) = \abs{T}-\rank(T)$.

  Let us consider the map $\bar\pi_{S,T} : Z(S) \to Z(T)$. By Lemma~\ref{Lemma:Ipi}, it is well-defined and surjective.
  Since a group homomorphism must map an element of finite order to another element of finite order,
  it is clear that $\bar\pi_{S,T}|_{\LG(S)} (\LG(S)) \subseteq \LG(T)$. Suppose that this inclusion is strict.
  Since $\bar\pi_{S,T}$ is surjective, this implies that
  an element of $Z(S)$ of infinite order is mapped to $\LG(S)$.
  But this is impossible: since $\bar\pi_{S,T}$ is surjective, by a dimension argument, the elements of infinite order of $Z(S)$ must be mapped to elements of infinite order of $Z(T)$.
 \end{proof}

 \begin{Lemma}
 \label{Lemma:Key}
   Let $S$, $T$ as above.
   Recall that $\bar\varphi_S : \LG(S) \to C(S) $ and $\bar\varphi_T : \LG(T) \to C(T)$ are bijections between the layer groups and the set of layers
   of the toric arrangement that are defined by $S$ and $T$, respectively. These  maps allow us to identify elements of the layer groups with the corresponding layers.

  \begin{enumerate}[(i)]
    \item  Suppose $ \rank(S) = \rank(T) $. Then $ p\in C(S) $ and $ q\in C(T)$ are identical (as subsets of the torus)
    if and only if $ \bar\pi_{S,T}(p) = q $.    
    \item  Suppose $\rank(S)>\rank(T)$. Then $p\in C(S)$ covers $ q\in C(T)$ 
    if and only if $\bar\pi_{S,T}(p) = q $.    
  \end{enumerate}

\end{Lemma}
\begin{proof}
 \begin{asparaenum}[(i)]
  \item 
   ``$\Rightarrow$''. 
   Let us consider liftings of the layers $p$ and $q$ to $\R^d$, \ie affine subspaces 
   $H(S,k_p)$ and $H(T,k_q)$ that are projected to $p$ and $q$ in the torus, respectively. Here $k_p\in W(S)\subseteq\Z^S$ and $k_q\in W(T)\subseteq\Z^T$ denote suitable vectors.
   By assumption, $p$ and $q$ are identical. Hence we are able to choose $k_p$ and $k_q$ \st 
   $ \bigcap_{j\in S} H( \{j\}, (k_p)_j )   = \bigcap_{j\in T} H( \{j\}, (k_q)_j )$. 
   Since each $H( \{j\}, (k_p)_j )$ is an affine linear space and $p\neq \emptyset$, this implies that $ (k_p)_j = (k_q)_j $ for all $ j \in T $
   and $ H( \{s\}, (k_p)_s ) \supseteq H(T,k_q) $,  where $ s $ denotes the unique element of $S\setminus T$.
   Hence $\pi_{S,T}(k_p) = k_q$. Since $k_p$ represents $p \in \LG(S)$
   and $k_q$ represents $q\in \LG(T)$, this implies
   $\bar\pi_{S,T}( p) = q$.

  ``$\Leftarrow$''.
  By definition, $p$  is a connected component of $\left(\bigcup_{k\in \Z^S} H(S,k) \right)/ \Z^d$.
  This means that for some $k_p \in \Z^S$, $p = H(S,k_p) + \Z^d$ holds.
  Let $k_q := \pi_{S,T}(k_p) \in W(T)$. 
  By assumption, $k_q$ represents $q\in \LG(T)$.
  Since $ T \subseteq S $, we have $ H(S,k_p) = H(T,k_q) \cap H(\{s\}, (k_p)_s) \subseteq  H(T,k_q)$. 
  On the other hand,  $ \rank(S) = \rank(T) $ 
   implies that the two affine spaces  have the same dimension. Hence
  we must have equality, \ie $ H(S,k_p) = H(T,k_q) $. By projecting this back to the torus, we obtain $ p = q $.
\medskip

\item 
 Since $\rank(S)=\rank(T)+1$, $p$ covers $q$ (as poset elements) if and only if 
 $p\subsetneq q$ (as subsets of the torus).

 ``$\Rightarrow$''. 
   As above, we  consider liftings of the layers $p$ and $q$ to $\R^d$, \ie affine subspaces 
   $H(S,k_p)$ and $H(T,k_q)$ that are projected to $p$ and $q$ in the torus, respectively. Again,  $k_p\in \Z^S$ and $k_q\in \Z^T$ 
   denote suitable vectors.
   By assumption, $p \subsetneq q$. Hence we are able to choose $k_p$ and $k_q$ \st 
   $ \bigcap_{j\in S} H( \{j\}, (k_p)_j ) = H(S, k_p) \subsetneq H( T, k_q )  = \bigcap_{j\in T} H( \{j\}, (k_q)_j )$. 
   Since each $H( \{j\}, (k_p)_j )$ is an affine linear space and $p\neq \emptyset$, this implies that $(k_p)_j= (k_q)_j$ for all $j\in T$
   and $ H( \{s\}, (k_p)_s ) \not\supseteq H(T,k_q) $,  where $ s $ denotes the unique element of $S\setminus T$.
   Hence $\pi_{S,T}(k_p) = k_q$. Since $k_p$ represents $p \in \LG(S)$
   and $k_q$ represents $q\in \LG(T)$, this implies
   $\bar\pi_{S,T}( p) = q$.

  ``$\Leftarrow$''.
  By definition, $p$  is a connected component of $\left(\bigcup_{k\in \Z^S} H(S,k) \right)/ \Z^d$.
  This means that for some integer $k_p$, $p = H(S,k_p) + \Z^d$ holds.
  Let $k_q:=\pi_{S,T}(k_p)$. 
  By assumption, $k_q$ represents $q\in \LG(T)$.
  Since $ T \subseteq S $, we have $ H(S,k_p) = H(T,k_q) \cap H(\{s\}, (k_p)_s) \subseteq  H(T,k_q)$. 
  On the other hand, since $ \rank(S) = \rank(T)+1 $, 
  the containment must be strict.
  By projecting this back to the torus, we obtain $p \subsetneq q$.
\qedhere
  \end{asparaenum}
\end{proof}

Now we are ready to prove the correctness of the algorithm.

\begin{proof}[Proof of Theorem~\ref{Theorem:ItWorks}]
  First note that the algorithm does indeed construct a directed acyclic graph:
  we can  assign to each vertex $v_{S,g}$ a rank, namely the matroid theoretic rank of $S$.
  Note that blue arcs always connect vertices of the same rank, while black arcs always point from a rank $k$ vertex to a rank $k+1$ vertex.
  This remains true after contracting the blue edges, hence the resulting graph contains no directed cycles.

  Now the result follows essentially from Lemma~\ref{Lemma:Key}.
  Two things need to be considered:
  \begin{enumerate}[(a)]
   \item  The algorithm computes a poset with the correct number of elements.
   \item The cover relations are correct.
  \end{enumerate}

  \begin{asparaenum}[(a)]
   \item 
  In the first line, the algorithm  creates all layers.
  But unless the matrix $X$ is totally unimodular, there will be a layer $L$ that is created more than once.
  This happens if there are two sets  $ S_1, S_2 \subseteq [N] $ \st $ L \in C(S_1) \cap C(S_2) $. 
  Given a layer $L$, the set of subsets of $[N]$ that define $L$
  can be partially ordered by inclusion and has a unique maximal element (the union of all such sets).
  This poset
  is canonically isomorphic to a subposet of $\Ccal(\Acal_T)$.
  By Lemma~\ref{Lemma:Key}.(i), the algorithm adds 
  for each layer $L$
  exactly the cover relations of this poset to the graph as blue edges, which are later contracted.
  Hence there will be exactly one vertex per layer.
  \item 
  We need to show 
  that 
  all cover relations constructed by the algorithm are indeed cover relations 
  in $\Ccal(\Acal_T)$ and that the algorithm  constructs all cover relations.
  The first statement follows directly from  Lemma~\ref{Lemma:Key}.(ii).

  Let $p\in C(S)$ and $q\in C(T)$, for suitable $S, T\subseteq [N]$. Let $u_p$ and $u_q$ denote the corresponding vertices in the graph $\Gcal$.
  We need to prove that it is sufficient to construct only the cover relations in $\Ccal(\Acal_T)$ that arise from pairs $(S,T)$
  with $T\subseteq S$ and $\abs{S\setminus T} = 1$, $\rank(S) = \rank(T) + 1$ and $\bar \pi_{S,T}(p)=q$.
  This is not completely straightforward since there \emph{are} cover relations with  
  $T\not \subseteq S$:
  for example, let $q$ be  (a quotient of) an affine line determined by three planes $H_1,H_2,H_3 \subseteq \R^3$
  and let $p$ be (a quotient of) a point in this line that is determined by ten planes in $\R^3$, including $H_1$ and $H_2$, but not $H_3$.

  Since $\Ccal(\Acal_T)$ is a graded poset,
  a layer defined by $S$ can only cover a layer defined by $T$ if $\rank(S)=\rank(T)+1$. 
  We also need that the flat spanned by $S$ contains the flat spanned by $T$ (otherwise,  the layer 
  defined by $S$ will not be a subset of the layer defined by $T$).
  This implies that there are independent sets $S'\subseteq S$ and $T'\subseteq T$ with $\rank(S')=\rank(S)$ and $\rank(T')=\rank(T)$ (hence $p\in C(S')$ and $q\in C(T')$).
  If $p$ covers $q$, by  Lemma~\ref{Lemma:Key}.(ii), the algorithm will add an arc between the vertices $w_p$ and $w_q$
  that correspond to $p\in C(S')$ and $q\in C(T')$, respectively. 
  By part (a) of this proof,
  $u_p$ and $w_p$, as well as $u_q$ and $w_q$ will be the same vertex after contracting the blue edges.
\qedhere
  \end{asparaenum}
\end{proof}

\renewcommand{\MR}[1]{} % no MR numbers

\bibliographystyle{amsplain}
\input{ComputingTheLayerPoset_arXiv_v1.bbl}

\end{document}

%% file: Bonn_HyperplaneArrangement.tikz
\tikzsetnextfilename{Bonn_HyperplaneArrangement}
    \begin{tikzpicture}[line join=round, scale=1.6]
%   \filldraw[fill=yellow!80!white,draw=black, thick](0,0) rectangle (1,1);
  \draw[bonn1colour, ultra thick](0,-1)--(0,1);
  \draw[bonn2colour, ultra thick](-1,0)--(1,0);
  \draw[bonn3colour, ultra thick](-0.9,-0.9)--(0.9,0.9);
  \draw[bonn4colour, ultra thick](-0.9,0.9)--(0.9,-0.9);
  \filldraw [black]  (0,0) circle (1.2pt); % node[anchor=north east]{$a$};
%  
%  
%   \filldraw [gray] (0,1) circle (1.2pt);
%   \filldraw [gray] (1,0) circle (1.2pt);
%   \filldraw [gray] (1,0) circle (1.2pt);
  \end{tikzpicture}

%% file: Bonn_ToricArrangement.tikz
\tikzsetnextfilename{Bonn_ToricArrangement}
% 
% you should put a \input{Bonn_style.tikz} somewhere
% 
    \begin{tikzpicture}[line join=round, scale=2.6]
  \filldraw[fill=yellow!60!white,draw=black, thick](0,0) rectangle (1,1);
  \draw[bonn2colour, ultra thick](0,0)--(1,0);
  \draw[bonn3colour, ultra thick](0,0)--(1,1);
  \draw[bonn4colour, ultra thick](1,0)--(0,1);
  \draw[bonn4colour, ultra thick](0,0.5)--(0.5,0);
  \draw[bonn4colour, ultra thick](0.5,1)--(1, 0.	5);
  \draw[bonn2colour, thick](0,1)--(1, 1);
  \draw[bonn1colour, ultra thick](0,0)--(0,1);
  \draw[bonn1colour, ultra thick](0.5,0)--(0.5,1);
  \draw[bonn1colour, thick](1,0)--(1,1);
%   \draw[draw=orange!85!black, ultra thick](0,1)--(1,0);
  \filldraw [black]  (0,0) circle (0.9pt) node[anchor=north east]{$0$};
%   \filldraw [black] (0.5,0.5)   circle (0.9pt) node[anchor=east]{$f$} ;
  \filldraw [black] (0.5,0) circle (0.9pt) node[anchor=south west]{$b$};
  \filldraw [black] (0,0.5) circle (0.9pt) node[anchor= west]{$a$};
  \filldraw [black] (0.25,0.25) circle (0.9pt) node[anchor=west]{$d$};
%   \filldraw [black] (0.333,0.333) circle (0.9pt) node[anchor=south east]{$g$};
  \filldraw [black] (0.75,0.75) circle (0.9pt) node[anchor= west]{$e$};
  \filldraw [black] (0.5,0.5) circle (0.9pt) node[anchor=west]{$c$};
%  Punkte fuer Kopien der vertices
  \filldraw [black!80!white] (1,0.5) circle (0.5pt); %node[anchor=west]{$d$};
  \filldraw [black!80!white] (0.5,1) circle (0.5pt); %node[anchor=west]{$d$};
  \filldraw [black!80!white] (1,0) circle (0.5pt); %node[anchor=west]{$d$};
  \filldraw [black!80!white] (0,1) circle (0.5pt); %node[anchor=west]{$d$};
  \filldraw [black!80!white] (1,1) circle (0.5pt); %node[anchor=west]{$d$};

  % 
% 
  %   \filldraw [black] (0.5,0.666) circle (0.9pt) node[anchor=south east]{$e$};
%   \filldraw [gray] (0,1) circle (1.2pt);
%   \filldraw [gray] (1,0) circle (1.2pt);
%   \filldraw [gray] (1,0) circle (1.2pt);
% Pfeile
  \draw[draw=black, thick, ->](-0.08, 0.05)--(-0.08, 0.95);
  \draw[draw=black, thick, ->](1.08, 0.05)--(1.08, 0.95);
  \draw[draw=black, thick, ->](0.05,-0.08)--(0.95,-0.08);
  \draw[draw=black, thick, ->](0.05,1.08)--(0.95,1.08);
  \end{tikzpicture}

%% file: Bonn_LatticeFlats.tikz
\tikzsetnextfilename{Bonn_LatticeFlats}
\providecommand{\setBonnLatticeFlatsScaling}{%
 \setlength{\BonnLatticeFlatsLength}{1cm} % laenge wird in Bonnstyle definiert
}
% zum Umstellen folgendes in die Datei setzen:
% \newcommand{\setBonnLatticeFlatsScaling}{%
%  \setlength{\BonnLatticeFlatsLength}{1cm} % laenge wird in Bonnstyle definiert
% }
% 
% 
% 
\setBonnLatticeFlatsScaling
% 
% 
%
% 
% http://tex.stackexchange.com/questions/47392/how-to-draw-a-poset-hasse-diagram-using-tikz
% 
% \scalebox{\BonnLatticeFlatsScaling}{
    \begin{tikzpicture}[line join=round, scale=2.6, >=latex'] % it seems that the scaling doesn't really matter here
      \node (top) at (0,0) {$X$};
      \node [below   of=top, xshift= -1.5\BonnLatticeFlatsLength] (left1)  {$\{x_1\}$};
      \node [below  of=top, xshift=-0.5\BonnLatticeFlatsLength ] (left2) {$\{x_2\}$};
      \node [below  of=top, xshift=0.5 \BonnLatticeFlatsLength] (right1) {$\{x_3\}$};
      \node [below  of=top, xshift=1.5 \BonnLatticeFlatsLength] (right2) {$\{x_4\}$};
      \node [below  of=top, yshift=-1 cm] (empty) {$\emptyset$};
      \draw [thick, <-] (top) -- (left1);
      \draw [thick, <-] (top) -- (left2);
      \draw [thick, <-] (top) -- (right1);
      \draw [thick, <-] (top) -- (right2);
      \draw [thick, <-] (left1) -- (empty);
      \draw [thick, <-] (left2) -- (empty);
      \draw [thick, <-] (right1) -- (empty);
      \draw [thick, <-] (right2) -- (empty);
  \end{tikzpicture}

%% file: Bonn_PosetLayers.tikz
\tikzsetnextfilename{Bonn_PosetLayers}
% 
% 
% http://tex.stackexchange.com/questions/47392/how-to-draw-a-poset-hasse-diagram-using-tikz
% 
% 
% 
\providecommand{\setBonnLayerPosetScaling}{%
  \setlength{\middleXPosetLayers}{-4.5cm} % Startpunkt der mittleren Reihe
  \setlength{\distAPosetLayers}{1.7cm}
  \setlength{\distBPosetLayers}{2.1cm}   
  \setlength{\topXPosetLayers}{-4.5cm} % Startpunkt der oberen Reihe
  \setlength{\topshiftXPosetLayers}{1.8cm}
}
\setBonnLayerPosetScaling
\newlength{\topYPosetLayers}
\setlength{\topYPosetLayers}{1.8cm} % wird zum Abstand der mittleren und oberen Ebene dazuaddiert
\newlength{\unterbrechungPosetLayers}
\setlength{\unterbrechungPosetLayers}{5pt} % manche linien in dieser Breite in weiß vormalen (um andere zu unterbrechen) benutze ich nicht
    \begin{tikzpicture}[line join=round, scale=2.6,, >=latex']
      \node (empty) at (0,0) {$\emptyset$};
%      middle layer
      \node [bonn1colour, above   of=empty, xshift=\middleXPosetLayers] (left1)  {$\{x_1\}$};
      \addtolength{\middleXPosetLayers}{\distAPosetLayers}
      \node [bonn1colour, above   of=empty, xshift=\middleXPosetLayers] (left2)  {$\{x_1\}$};
      \addtolength{\middleXPosetLayers}{\distBPosetLayers}
      \node [bonn2colour, above   of=empty, xshift=\middleXPosetLayers] (lcentre)  {$\{x_2\}$};
      \addtolength{\middleXPosetLayers}{\distBPosetLayers}
      \node [bonn3colour, above   of=empty, xshift=\middleXPosetLayers] (rcentre)  {$\{x_3\}$};
      \addtolength{\middleXPosetLayers}{\distBPosetLayers}
      \node [bonn4colour, above   of=empty, xshift=\middleXPosetLayers] (right1)  {$\{x_4\}$};
      \addtolength{\middleXPosetLayers}{\distAPosetLayers}
      \node [bonn4colour, above   of=empty, xshift=\middleXPosetLayers] (right2)  {$\{x_4\}$};
%       top layer
      \node [ above   of=empty, xshift=\topXPosetLayers, yshift=\topYPosetLayers] (l1r2)  {$a$};
      \addtolength{\topXPosetLayers}{\topshiftXPosetLayers}      
      \node [above   of=empty, xshift=\topXPosetLayers, yshift=\topYPosetLayers] (l2lcr2)  {$b$};
      \addtolength{\topXPosetLayers}{\topshiftXPosetLayers}      
      \node [above of=empty, xshift=\topXPosetLayers, yshift=\topYPosetLayers] (l123r1)  {$0$};
      \addtolength{\topXPosetLayers}{\topshiftXPosetLayers}      
      \node [ above   of=empty, xshift=\topXPosetLayers, yshift=\topYPosetLayers] (middle)  {$c$};
      \addtolength{\topXPosetLayers}{\topshiftXPosetLayers}      
      \node [ above   of=empty, xshift=\topXPosetLayers, yshift=\topYPosetLayers] (rcr2a)  {$d$};
      \addtolength{\topXPosetLayers}{\topshiftXPosetLayers}      
      \node [ above   of=empty, xshift=\topXPosetLayers, yshift=\topYPosetLayers] (rcr2b)  {$e$};
% 
%	lower lines
      \draw [thick, ->] (empty) -- (left1);
      \draw [thick, ->] (empty) -- (left2);
      \draw [thick, ->] (empty) -- (lcentre);
      \draw [thick, ->] (empty) -- (rcentre);
      \draw [thick, ->] (empty) -- (right1);%     
      \draw [thick, ->] (empty) -- (right2);%     
% %     
% the origin
      \draw [thick, ->] (left1) -- (l123r1);
      \draw [thick, ->] (lcentre) -- (l123r1);
      \draw [thick, ->] (rcentre) -- (l123r1);
      \draw [thick, ->] (right1) -- (l123r1);
      \draw [thick, ->] (left2) -- (middle);
      \draw [thick, ->] (rcentre) -- (middle);
      \draw [thick, ->] (right1) -- (middle);
      \draw [thick, ->] (left2) -- (l2lcr2);
      \draw [thick, ->] (lcentre) -- (l2lcr2);
      \draw [thick, ->] (right2) -- (l2lcr2);
      \draw [thick, ->] (left1) -- (l1r2);
      \draw [thick, ->] (right2) -- (l1r2);
      \draw [thick, ->] (rcentre) -- (rcr2a);
      \draw [thick, ->] (right2) -- (rcr2a);
      \draw [thick, ->] (rcentre) -- (rcr2b);
      \draw [thick, ->] (right2) -- (rcr2b);
% 
% 
%       \draw [thick] (centre) -- (12);
%       \draw %[preaction={draw=white, -,line width=\unterbrechung}]
%         [thick] (left1) -- (13);
%       \draw [thick] (centre3) -- (13);
%       \draw [thick] (left2) -- (21);
%       \draw [thick] (centre1) -- (21);
%       \draw [thick] (left2) -- (22);
%       \draw [thick] (centre2) -- (22);
%       \draw [thick] (left2) -- (23);
%       \draw [thick] (centre3) -- (23);
%       \draw [thick] (left2) -- (2R);
%       \draw [thick] (right) -- (2R);
%       \draw [thick] (centre2) -- (2mR);
%       \draw [thick] (right) -- (2mR);
%       \draw [thick] (centre3) -- (3mR);
%       \draw [thick] (right) -- (3mR);
      \end{tikzpicture}

%% file: ComputingTheLayerPoset_arXiv_v1.bbl
\providecommand{\bysame}{\leavevmode\hbox to3em{\hrulefill}\thinspace}
\providecommand{\MR}{\relax\ifhmode\unskip\space\fi MR }
\providecommand{\MRhref}[2]{%
  \href{http://www.ams.org/mathscinet-getitem?mr=#1}{#2}
}
\providecommand{\href}[2]{#2}